\documentclass[11pt]{article}
\usepackage[T1]{fontenc}

\usepackage{authblk}
\usepackage{amsmath,amssymb,mathrsfs}
\usepackage{enumitem}
\usepackage{verbatim}
\usepackage{color}
\usepackage{pifont}
\usepackage[margin=1in]{geometry}
\usepackage{setspace}

\newcommand{\lang}{\mathcal{L}}
\newcommand{\limp}{\longrightarrow}
\newcommand{\pow}{\mathcal{P}}

\newcommand{\hollowstar}{\text{\ding{73}}}
\newcommand{\fstar}{\text{\ding{72}}}
\newcommand{\ci}{\mathbf{Ci}}
\newcommand{\mbd}{\mathbf{mbD}}
\newcommand{\mbst}{\mathbf{mb}^\fstar}
\newcommand{\vdashs}{\vdash_{\mathcal{S}}}
\newcommand{\vdashmb}{\vdash_{\mbst}}
\newcommand{\modelsmb}{\models_{\mbst}}
\newcommand{\cneg}{\sim}
\newcommand{\probsem}{\Vdash_\mathbb{P}}
\newcommand{\diam}{\blacklozenge}

\usepackage{amsthm}
\theoremstyle{definition}
\newtheorem{theorem}{Theorem}[section]
\newtheorem{lemma}[theorem]{Lemma}
\newtheorem{corollary}[theorem]{Corollary}
\newtheorem{definition}[theorem]{Definition}
\newtheorem{remark}[theorem]{Remark}
\newtheorem{example}[theorem]{Example}

\usepackage[backref=page,colorlinks=true,allcolors=blue]{hyperref}

  \title{Paracomplete Probabilities\footnote{The final version of this paper has been submitted for publication.}}
  \author[1]{Sankha S. Basu}
  \author[1]{Esha Jain}
  \date{
  August 06, 2026}
  \affil[1]{Department of Mathematics\\
  Indraprastha Institute of Information Technology-Delhi\\
  New Delhi, India.}

\begin{document}
%
\maketitle              

\begin{abstract}
This paper presents an advance in the direction of working with probabilities in a paracomplete setting using Logics of Formal Undeterminedness (LFUs). The undeterminedness is interpreted here as missing evidence. A theorem of total paracomplete probability and a paracomplete Bayes' rule have been proved using this setup. We end with a definition of a paracomplete probability space illustrating a way to define probabilities on sets in the presence of undeterminedness.
\end{abstract}


\section{Introduction}
A logic $\langle\lang,\vdash\rangle$ is said to be \emph{paracomplete} if the \emph{law of excluded middle (LEM)} fails in it, i.e., there exists $\alpha\in\lang$ such that $\not\vdash\alpha\lor\neg\alpha$. In a paracomplete logic, two propositions $\alpha$ and $\neg\alpha$ can both be false. A paracomplete logic is dual to a paraconsistent logic, where the \emph{law of explosion ECQ (ex contradictione sequitor quodlibet)} fails, i.e., there exists $\alpha,\beta\in\lang$ such that
$\{\alpha,\neg\alpha\}\nvdash\beta$. There are more ways of describing paracompleteness (see \cite{daCosta1986}). Intuitionistic propositional logic (IPC) is an example of a paracomplete logic.

A \emph{logic of formal undeterminedness (LFU)} is a paracomplete logic that has a primitive or defined unary \emph{determinedness} operator, usually denoted by $\hollowstar$, that can recover LEM in a paracomplete setting. Intuitively speaking, in an LFU $\langle\lang,\vdash\rangle$ $\hollowstar\alpha$ is meant to assert that ``$\alpha$ is determined,'' and hence, in the presence of $\hollowstar\alpha$, either $\alpha$ or $\neg\alpha$ will be the case. In other words, $\hollowstar\alpha\vdash\alpha\lor\neg\alpha$, for all $\alpha\in\lang$. A unary \emph{undeterminedness} operator, usually denoted by $\fstar$, can then be defined as the classical negation of $\hollowstar$, i.e., for any $\alpha\in\lang$, $\fstar\alpha:=\,\cneg\hollowstar\alpha$ (`$\cneg$' being the classical negation operator). Alternatively, an LFU can be described with an undeterminedness operator in the language and the determinedness operator can be defined in terms of it using the classical negation. The interdefinability between the determinedness and undeterminedness operators is, of course, only possible if a classical negation is available or definable. Intuitively speaking, in an LFU $\langle\lang,\vdash\rangle$, $\fstar\alpha$ is meant to assert that ``there is missing evidence about $\alpha$,'' or equivalently, ``$\alpha$ is undetermined.'' A formal definition of an LFU in terms of determinedness can be found in \cite{marcos2005nearly,Szmuc2016} and also in \cite{carnielli2020recovery}. We have included a formal definition of an LFU described with an undeterminedness operator $\fstar$ in Section 2; this is adapted from the one in \cite{carnielli2020recovery}.

LFUs were introduced in \cite{marcos2005nearly} as dual to the \emph{logics of formal inconsistency (LFIs)}. However, as mentioned in \cite{carnielli2020recovery}, the idea can be traced back to \cite{daCosta1986}. A \emph{logic of formal inconsistency (LFI)} is a paraconsistent logic that has a primitive or defined unary consistency operator, usually denoted by $\circ$, which is used to recover the principle of explosion in a paraconsistent setting. Intuitively speaking, in an LFI $\langle\lang,\vdash\rangle$, for any $\alpha\in\lang$, $\circ\alpha$ is meant to assert that ``$\alpha$ is consistent,'' and hence, $\{\alpha,\neg\alpha,\circ\alpha\}\vdash\beta$, for all $\beta\in\lang$ (see \cite{CarnielliConiglio2016,Carnielli2007,Carnielli2015} for more on LFIs).

In an LFI or LFU, the semantic values $T/F$ or $1/0$ are not  used to convey truth and falsity, but rather the presence and absence of evidence, respectively. 

The system $\ci$ is an LFI that is used in \cite{bueno2016paraconsistent} to define a generalized notion of probability and prove a theorem of total probability for the contradictory evidence case. In this paper, we explore the dual case of missing evidence using an LFU and follow the work in \cite{bueno2016paraconsistent} to generalize the theorem of total probability for this case. In order to do this, we introduce a logic that we call $\mbst$ with a primitive unary undeterminedness operator $\fstar$. A semantics for this logic and the soundness and completeness results with respect to this semantics are also presented. The logic $\mbst$, however, turns out to be the LFU $\mbd$ (see \cite{carnielli2020recovery}) in disguise which was described with a primitive unary determinedness operator $\hollowstar$. 
 
We next define probability functions that can be used to provide a semantics for $\mbst$. The completeness of $\mbst$ with respect to this probability semantics is then established using the completeness result mentioned above. In the final section of the article, we have defined a paracomplete probability space in an attempt to handle the concept of probabilities on sets in the presence of undeterminedness. This work can also be seen in connection with some other work on probability functions relative to intuitionistic logic (not LFUs) that were developed in \cite{morgan1983probabilistic} and \cite{weatherson2003classical}.

\section{The logic \texorpdfstring{$\mbst$}{mb*}}

\subsection{LFUs formally}
Let $\lang$ be the set of formulas generated inductively over a denumerable set of variables $V$ using a finite set of connectives or operators, called the \emph{signature}, $\Sigma$. A logic with the signature $\Sigma$ is then a pair $\mathcal{S}=\langle\lang,\vdashs\rangle$, where $\vdashs\,\subseteq\pow(\lang)\times\lang$ is the consequence relation of the logic $\mathcal{S}$.

\begin{definition}
    A logic $\mathcal{S}=\langle\lang,\vdashs\rangle$ is called \emph{Tarskian} if it satisfies the following properties. For every $\Gamma\cup\Delta\cup\{\alpha\}\subseteq\lang$,
    \begin{enumerate}[label=(\roman*)]
        \item if $\alpha\in\Gamma$, then $\Gamma\vdashs\alpha$ (\emph{Reflexivity});
        \item if $\Gamma\vdashs\alpha$ and $\Gamma\subseteq\Delta$, then $\Delta\vdashs\alpha$ (\emph{Monotonicity});
        \item if $\Delta\vdashs\alpha$ and $\Gamma\vdashs\beta$ for all $\beta\in\Delta$, then $\Gamma\vdashs\alpha$ (\emph{Transitivity/ Cut}).
    \end{enumerate}
    $\mathcal{S}$ is said to be \emph{finitary} if for every $\Gamma\cup\{\alpha\}\subseteq\lang$, $\Gamma\vdashs\alpha$ implies that there exists a finite $\Gamma_0\subseteq\Gamma$ such that $\Gamma_0\vdashs\alpha$.

    $\mathcal{S}$ is said to be \emph{structural} if for every $\Gamma\cup\{\alpha\}\subseteq\lang$, $\Gamma\vdashs\alpha$ implies that $\sigma[\Gamma]\vdashs\sigma(\alpha)$, for every substitution $\sigma$ of formulas for variables.
\end{definition}

\begin{remark}
    The set of formulas $\lang$ can also be described as the formula algebra over $V$ of some type/ signature. The formula algebra has the universal mapping property for the class of all algebras of the same type as $\lang$ over $V$, i.e., any function $f:V\to A$, where $A$ is the universe of an algebra $\mathbf{A}$ of the same type as $\mathcal{L}$, can be uniquely extended to a homomorphism from $\lang$ to $\mathbf{A}$ (see \cite{FontJansanaPigozzi2003,Font2016} for more details).
    
    A \emph{substitution} can then be defined as any function $\sigma:V\to\lang$ that extends to a unique endomorphism (also denoted by $\sigma$) from $\lang$ to itself via the universal mapping property. The logic $\mathcal{S}$ is then defined to be structural as above.
\end{remark}

\begin{definition}\label{def:LFU}
    Let $V$ be as before, $\Sigma=\{\land,\lor,\limp,\neg,\fstar\}$ be a signature, where $\neg,\fstar$ are unary and the rest are binary connectives. Suppose $\mathcal{S}=\langle\lang,\vdashs\rangle$ is a Tarskian, finitary and structural logic where $\lang$ is the formula algebra generated over $V$ with the signature $\Sigma$. Moreover, suppose $\lor$ enjoys the following standard property of a disjunction. For any $\Gamma\cup\{\alpha,\beta,\gamma\}\subseteq\lang$, 
    \[
    \Gamma\cup\{\alpha\}\vdashs\gamma\hbox{ and }\Gamma\cup\{\beta\}\vdashs\gamma\hbox{ iff }\Gamma\cup\{\alpha\lor\beta\}\vdashs\gamma.
    \]
    $\mathcal{S}$ is then called an \emph{LFU (with respect to $\neg$ and $\fstar$)} if the following conditions hold.
    \begin{enumerate}[label=(\roman*)]
        \item $\not\vdashs\alpha\lor\neg\alpha$, for some $\alpha\in\lang$.
        \item There exists $\alpha\in\lang$ such that
        \begin{enumerate}
            \item $\not\vdashs\alpha\lor\fstar\alpha$, and
            \item $\not\vdashs\neg\alpha\lor\fstar\alpha$.
        \end{enumerate}
        \item $\vdash\alpha\lor\neg\alpha\lor\fstar\alpha$, for all $\alpha\in\lang$.
    \end{enumerate}
\end{definition}

\begin{remark}
    The above definition of an LFU is adapted from the one in \cite{carnielli2020recovery}, where an LFU is defined using a signature containing a determinedness operator $\hollowstar$. The operator $\fstar$ in the signature used here is, on the other hand, the undeterminedness operator.
\end{remark}

\subsection{Syntax of the logic \texorpdfstring{$\mbst$}{mb*}}\label{subsec:syntax}
The logic $\mbst=\langle\lang,\vdashmb\rangle$, where $\lang$ is the formula algebra generated over a denumerable set of variables $V$ using a signature $\Sigma=\{\land,\lor,\limp,\neg,\fstar\}$, is the logic induced by the following Hilbert-style presentation. 

\textbf{Axiom Schema:}

\begin{enumerate}
    \item $\alpha\limp(\beta\limp\alpha)$
    \item $(\alpha\limp\beta)\limp((\alpha\limp(\beta\limp\gamma))\limp(\alpha\limp\gamma))$
    \item $\alpha\limp (\beta\limp(\alpha\land\beta))$
    \item $(\alpha\land\beta)\limp\alpha$
    \item $(\alpha\land\beta)\limp\beta$
    \item $\alpha\limp(\alpha\lor\beta)$
    \item $\beta\limp(\alpha\lor\beta)$
    \item $(\alpha \limp\gamma)\limp((\beta\limp\gamma)\limp((\alpha\lor\beta)\limp\gamma))$
    \item $\alpha\lor(\alpha\limp\beta)$
    \item $\alpha\limp(\neg\alpha\limp\beta)$ (explosion)
    \item $\alpha\lor\neg\alpha\lor\fstar\alpha$ (\emph{included middle})
\end{enumerate}

\textbf{Inference Rule} 

$\dfrac{\alpha,\;\alpha\limp\beta}{\beta}$ [Modus ponens (MP)]

\begin{remark}\label{rem:axioms}
    The axioms (1)--(9) in the above Hilbert-style presentation of $\mbst$ are the axioms of positive classical propositional logic (CPL$^+$). As mentioned earlier, the logic $\mbst$ is the logic $\mbd$ in disguise, which was defined using a signature containing a determinedness operator $\hollowstar$ in \cite{carnielli2020recovery}. The explosion axiom, Axiom (10), is also an axiom of $\mbd$. The axiom of included middle, Axiom (11), is a variant of the axiom of `gentle excluded middle (GPEM)' of $\mbd$.
\end{remark}

The fact that $\mbst$ is induced by a Hilbert-style presentation, i.e., it is a Hilbert-style logic is captured in the following definitions of \emph{syntactic derivation} and \emph{syntactic entailment} in $\mbst$.

\begin{definition} 
Let $\Gamma\cup\{\varphi\}\subseteq\lang$. A \emph{derivation} of $\varphi$ from $\Gamma$ in $\mbst$ is a finite sequence $(\varphi_{1},\ldots,\varphi_{n})$ of elements in $\lang$, where $\varphi_{n}=\varphi$ and for each $1\le i\le n$, 

\begin{enumerate}[label=(\roman*)]
    \item $\varphi_{i}$ is an instance of an axiom of $\mbst$, or
    \item $\varphi_{i}\in \Gamma$, or
    \item there exist $1\le j,k<i$ such that $\varphi_{i}$ is obtained from $\varphi_{j},\varphi_{k}$ by MP.
\end{enumerate}
We say that $\varphi$ is \emph{syntactically derivable} or \emph{syntactically entailed} from $\Gamma$, and write $\Gamma\vdashmb\varphi$, if there is a derivation of $\varphi$ from $\Gamma$.
\end{definition}

\begin{remark}\label{rem:mbst-Tarskian/finitary}
    It follows from the above definition of a syntactic consequence in $\mbst$, that $\mbst$ is a Tarskian, finitary and structural logic. 
\end{remark}

\begin{theorem}\label{thm:DedThm}
    The Deduction theorem holds in $\mbst=\langle\lang,\vdashmb\rangle$, i.e., for any $\Gamma\cup\{\alpha,\beta\}\subseteq\lang$, $\Gamma\cup\{\alpha\}\vdashmb\beta$ iff $\Gamma\vdashmb\alpha\limp\beta$.
\end{theorem}

\begin{proof}
This is standard in the presence of Axioms 1 and 2 and MP as the only rule of inference.
\end{proof}

\subsection{Semantics for the logic \texorpdfstring{$\mbst$}{mb*}}\label{subsec:semantics}
\begin{definition}\label{def:semantics}
Let $\lang$ be the set of formulas of $\mbst$ as described above. A function $v:\lang\to\{0,1\}$ is an \emph{$\mbst$-valuation} if it satisfies the following conditions.
\begin{enumerate}[label=(\roman*)]
    \item $v(\alpha\lor\beta)=1$ iff $v(\alpha)=1$ or $v(\beta)=1$.
    \item $v(\alpha \land\beta)=1$ iff $v(\alpha)=1$ and $v(\beta)=1$.
    \item $v(\alpha\limp\beta)=1$ iff $v(\alpha)=0$ or $v(\beta)=1$.
    \item If $v(\alpha)=1$, then $v(\neg\alpha)=0$.
    \item If $v(\alpha)=0=v(\neg\alpha)$, then $v(\fstar\alpha)=1$.
\end{enumerate} 
\end{definition}

\begin{remark}\label{rem:val}
We note that the conditions (i)--(iii) above are exactly the truth conditions for the classical connectives $\land,\lor,\limp$. However, $\mbst$-valuations are not truth functional due to conditions (iv) and (v). Thus, any $\mbst$-valuation behaves like a classical valuation for $\neg\,$- and $\fstar\,$- free formulas, i.e., for all CPL$^+$-formulas.
\end{remark}

\begin{definition}
Suppose $\Gamma\cup\{\varphi\}\subseteq\lang$, where $\lang$ is the set of formulas of $\mbst$. Then, we say that $\Gamma$ \emph{(semantically) entails} $\varphi$, and write $\Gamma\modelsmb\varphi$, if for every $\mbst$-valuation $v$, $v(\varphi)=1$ whenever $v(\gamma)=1$ for every $\gamma\in\Gamma$.

A formula $\varphi$ is an \emph{$\mbst$-tautology}, if $v(\varphi)=1$ for every $\mbst$-valuation $v$.
\end{definition}

\begin{remark}
    As mentioned in \cite{Carnielli2015}, the truth values 1 and 0 are not to be read as true and false, respectively, but instead as presence and absence of evidence, respectively. Thus, for any formula $\alpha$ and $\mbst$-valuation $v$,
    \begin{itemize}
        \item $v(\alpha)=1$ means there is  evidence that $\alpha$ is true;
        \item $v(\alpha)=0$ means there is no evidence that $\alpha$ is true;
        \item $v(\neg\alpha)=1$ means there is  evidence that $\alpha$ is false;
        \item $v(\neg\alpha)=0$ means there is no evidence that $\alpha$ is false; and
        \item $v(\fstar\alpha)=1$ means evidence is missing for $\alpha$.  
    \end{itemize}

    It may be noted that it is possible to have $v(\alpha)=1=v(\fstar\alpha)$ or $v(\neg\alpha)=1=v(\fstar\alpha)$. In the first case, there is evidence for the truth of $\alpha$ but the evidence is not complete or conclusive, and hence $v(\fstar\alpha)=1$. Similarly, in the second situation, there is evidence for the falsity of $\alpha$, but there is still some ambiguity. See Example \ref{exm:paracompleteprob} for a real life scenario.
\end{remark}

We can write truth tables for the formulas of $\mbst$ as exemplified below. These are, of course, of non-deterministic character whenever a formula involves $\neg$ or $\fstar$.

\begin{enumerate}[label=(\arabic*)]
\item $\begin{array}{|c|c|c|c|c|c|}
\hline
\alpha&\neg\alpha&\alpha\land\neg\alpha&\neg\neg\alpha&\alpha\limp\neg\neg\alpha& \neg\neg\alpha\limp\alpha\\
\hline
   1&0&0&1&1&1\\\cline{4-6}
   &&&0&0&1\\
   \hline
   0&1&0&0&1&1\\\cline{2-6}
   &0&0&1&1&0\\\cline{4-6}
   &&&0&1&1\\
   \hline
\end{array}$

\item $\begin{array}{|c|c|c|c|c|c|}
\hline
\alpha&\neg\alpha&\fstar\alpha&\alpha\lor\neg\alpha&\neg(\alpha\lor\neg\alpha) &\neg(\alpha\lor\neg\alpha)\limp\fstar\alpha\\
\hline
   1&0&1&1&0&1\\\cline{3-3}
   &&0&&&\\
   \hline
   0&1&1&1&0&1\\\cline{3-3}
   &&0&&&\\
   \cline{2-5}
   &0&1&0&1&\\\cline{5-5}
   &&&&0&\\
   \hline
\end{array}$

\item $\begin{array}{|c|c|c|c|c|c|c|}
\hline
\alpha&\neg\alpha&\alpha\lor\neg\alpha&\fstar\alpha&\alpha\lor\fstar\alpha &\neg\alpha\lor\fstar\alpha &\alpha\lor\neg\alpha\lor\fstar\alpha\\
\hline
   1&0&1&1&1&1&1\\\cline{4-7}
   &&&0&1&0&1\\
   \hline
   0&1&1&1&1&1&1\\\cline{4-7}
   &&&0&0&1&1\\
   \cline{2-7}
   &0&0&1&1&1&1\\
   \hline
\end{array}$
\end{enumerate}

\begin{remark}\label{rem:tt}
    We have the following observations from the above tables.
\begin{enumerate}[label=(\roman*)]
    \item Table (1) shows that the classical double negation rules do not hold in $\mbst$.
    \item Table (1) also shows that for any formula $\alpha$, $\alpha\land\neg\alpha$ always receives the value 0. However, this does not imply that for any formula $\alpha$, $\neg(\alpha\land\neg\alpha)$ is an $\mbst$-tautology.
    \item Table (2) shows that that $\alpha\lor\neg\alpha$ is not an $\mbst$-tautology, which is expected for a paracomplete logic. It also shows that $\neg(\alpha\lor\neg\alpha)\limp\fstar\alpha$ is an $\mbst$-tautology, for any $\alpha$. This can be translated to: ``if there is evidence that $\alpha\lor\neg\alpha$ is false, for some $\alpha$, i.e., the law of excluded middle fails for this $\alpha$, then there must be evidence missing for $\alpha$.'' However, it can observed from the table that $\fstar\alpha\limp(\alpha\lor\neg\alpha)$ is not an $\mbst$-tautology, as it possible for $\fstar\alpha$ to assume the value 1, while $\alpha\lor\neg\alpha$ is mapped to 0.
    \item From table (3), we see that, in addition to $\alpha\lor\neg\alpha$, $\alpha\lor\fstar\alpha$ and $\neg\alpha\lor\fstar\alpha$ are also not $\mbst$-tautologies. However, $\alpha\lor\neg\alpha\lor\fstar\alpha$ is an $\mbst$-tautology for any $\alpha$, i.e., included middle holds. Upon proving the Completeness theorem, these would suffice to establish that $\mbst$ is indeed an LFU, as per Definition \ref{def:LFU}.
\end{enumerate}  
\end{remark}

\subsection{Soundness and Completeness}
In this subsection, we embark on the task of proving the soundness and completeness of $\mbst$ as presented in Subsection \ref{subsec:syntax} with respect to the semantics presented in Subsection \ref{subsec:semantics}. As before, let $\mbst=\langle\lang,\vdashmb\rangle$ and $\modelsmb\,\subseteq\pow(\lang)\times\lang$ be as described above.

\begin{theorem}[Soundness]
Suppose $\Gamma\cup\{\varphi\}\subseteq\mathcal{L}$. Then, $\Gamma\vdash_{\mbst}\varphi$ implies $\Gamma\models_{\mbst}\varphi$
\end{theorem}

\begin{proof}
Suppose $\Gamma\vdash_{\mbst}\varphi$. Then, there exists a derivation $(\varphi_1,\ldots,\varphi_n=\varphi)$ of $\varphi$ from $\Gamma$. Suppose $v:\lang\to\{0,1\}$ is an $\mbst$-valuation such that $v(\gamma)=1$ for all $\gamma\in\Gamma$. We now show that $v(\varphi_i)=1$ for each $1\le i\le n$ by induction on $i$.

\textsc{Base Case:} $i=1$.

Clearly, $\varphi_{1}$ must be an instance of an axiom or $\varphi_{1}\in\Gamma$.
If $\varphi_{1}\in \Gamma$, then $v(\varphi_1)=1$ by hypothesis.

Suppose $\varphi_{1}$ is an instance of an axiom. If $\varphi_1$ is an instance of one of the Axioms (1)--(9), i.e., an axiom of CPL$^{+}$, as noted in Remark \ref{rem:axioms}, then, by Remark \ref{rem:val}, $v(\varphi_1)=1$.

If $\varphi_1$ is an instance of Axiom (10), then $\varphi_1=\alpha\limp(\neg\alpha\limp\beta)$ for some $\alpha,\beta\in\lang$. Now, if $v(\alpha)=0$, then $v(\varphi_1)=1$. On the other hand, if $v(\alpha)=1$, then, by condition (iv) of Definition \ref{def:semantics}, $v(\neg\alpha)=0$, which implies that $v(\neg\alpha\limp\beta)=1$. Hence, $v(\varphi_1)=1$.

Finally, if $\varphi_1$ is an instance of Axiom (11), then $\varphi_1=\alpha\lor(\neg\alpha\lor\fstar\alpha)$ for some $\alpha\in\lang$. We note that, if $v(\alpha)=1$ or $v(\neg\alpha)=1$, then $v(\varphi_1)=1$. On the other hand, if $v(\alpha)=0=v(\neg\alpha)$, then, by condition (v) of Definition \ref{def:semantics}, $v(\fstar\alpha)=1$, and hence, $v(\varphi_1)=1$. 

Thus, in all cases $v(\varphi_1)=1$.

\textsc{induction hypothesis:} Suppose $v(\varphi_l)=1$ for all $1\le l<i$ for some $1<i\le n$.

\textsc{Induction step:} We need to show that $v(\varphi_i)=1$. In case $\varphi_i$ is an instance of an axiom or $\varphi_i\in\Gamma$, then $v(\varphi_i)=1$ by the same arguments as in the base case. Otherwise, there exist $j,k\le i$ such that $\varphi_i$ is obtained from $\varphi_j,\varphi_k$ by MP. Without loss of generality, we can assume that $\varphi_{k}=\varphi_{j}\limp\varphi_i$. By the induction hypothesis, $v(\varphi_j)=1=v(\varphi_k)=v(\varphi_j\limp\varphi_i)$. Thus, $v(\varphi_i)=1$. 

Hence, $v(\varphi_i)=1$ for all $1\le i\le n$, and hence, in particular, $v(\varphi_n)=v(\varphi)=1$. Thus, $\Gamma\models_\mbst\varphi$.
\end{proof}

We next invoke some general tools below to help us prove the completeness theorem. These are the same ones that have been used for proving the completeness of LFIs in \cite{Carnielli2007}, for example.

\begin{definition}
    Suppose $\mathcal{S}=\langle\lang,\vdashs\rangle$ is a logic and $\Gamma\cup\{\varphi\}\subseteq \lang$. $\Gamma$ is called \emph{maximal relative to $\varphi$} in $\mathcal{S}$, if $\Gamma\not\vdashs\varphi$ but for any $\psi\in\lang\setminus\Gamma$, $\Gamma\cup\{\psi\}\vdashs\varphi$.
\end{definition}

\begin{definition}
    Suppose $\mathcal{S}=\langle\lang,\vdashs\rangle$ is a logic and $\Gamma\cup\{\psi\}\subseteq \lang$. $\Gamma$ is called \emph{closed} in $\mathcal{S}$ or a \emph{closed theory} of $\mathcal{S}$, if $\Gamma\vdashs\psi$ iff $\psi\in\Gamma$ for all $\psi\in\lang$.
\end{definition}

\begin{lemma}\label{lem:maxrel->closed}
    Suppose $\mathcal{S}=\langle\lang,\vdashs\rangle$ is a Tarskian logic and $\Gamma\subseteq \lang$. If $\Gamma$ is maximal relative to $\varphi$, for some $\varphi\in\lang$, then $\Gamma$ is closed.
\end{lemma}

\begin{proof}
    Suppose $\Gamma$ is maximal relative to $\varphi$, but is not closed. Now, as $\mathcal{S}$ is Tarskian, and hence, satisfies reflexivity, $\Gamma\vdashs\psi$ for all $\psi\in\Gamma$. This implies that, since $\Gamma$ is not closed, there exists $\psi\in\lang$ such that $\Gamma\vdashs\psi$ but $\psi\notin\Gamma$. Then, since $\Gamma$ is maximal relative to $\varphi$, $\Gamma\not\vdashs\varphi$ but $\Gamma\cup\{\psi\}\vdashs\varphi$. Now, by using reflexivity and that $\Gamma\vdashs\psi$, we can conclude that $\Gamma\vdashs\theta$ for all $\theta\in\Gamma\cup\{\psi\}$. So, by transitivity and that $\Gamma\cup\{\psi\}\vdashs\varphi$, $\Gamma\vdashs\varphi$. This is a contradiction. Hence, $\Gamma$ must be closed.
\end{proof}

The following lemma is the well-known Lindenbaum-Asser theorem and can be found, e.g., in \cite{Beziau1999}. We skip the proof.

\begin{lemma}\label{lem:Lind-Asser}
    Suppose $\mathcal{S}=\langle\lang,\vdashs\rangle$ is a Tarskian and finitary logic over the language $\mathcal{L}$. Let $\Gamma \cup\{\varphi\}\subseteq\mathcal{L}$ be such that $\Gamma\not\vdashs\varphi$. Then, there exists $\Delta\supseteq\Gamma$ that is maximal relative to $\varphi$ in $\mathcal{S}$.  
\end{lemma}

We now come back to the logic $\mbst$. Let $\lang$ be the set of formulas of $\mbst$, as before.

\begin{lemma}\label{lem:mbst-val}
Let $\Gamma\cup\{\varphi\}\subseteq\lang$ such that $\Gamma$ is maximal relative to $\varphi$ in $\mbst$. Then, the mapping $v:\lang\to\{0,1\}$ defined by
\[
v(\psi)=1\hbox{ iff }\psi\in\Gamma,\hbox{ for all }\psi\in\lang,
\]
is an $\mbst$-valuation.
\end{lemma}

\begin{proof} 
We first note that $\mbst$ is a Tarskian logic as pointed out in Remark \ref{rem:mbst-Tarskian/finitary}. Thus, by Lemma \ref{lem:maxrel->closed}, $\Gamma$ being maximal relative to $\varphi$ in $\mbst$, is closed. We now show below that the function $v$ defined above satisfies the conditions in Definition \ref{def:semantics}. Let $\alpha,\beta\in\lang$.
\begin{enumerate}[label=(\roman*)]
    \item Suppose $v(\alpha\lor\beta)=1$. Then, by definition of $v$, $\alpha\lor\beta\in\Gamma$. If possible, let $v(\alpha)=0$ and $v(\beta)=0$, i.e., $\alpha,\beta\notin\Gamma$. 
    
    Now, since $\Gamma$ is maximal relative to $\varphi$, $\Gamma\cup\{\alpha\}\vdashmb\varphi$ and  $\Gamma\cup\{\beta\}\vdashmb\varphi$. This implies, by the Deduction theorem (Theorem \ref{thm:DedThm}), $\Gamma\vdashmb\alpha \limp\varphi$ and $\Gamma\vdashmb\beta \limp\varphi$. Since $\Gamma$ is closed, $\alpha\limp\varphi,\beta\limp\varphi\in\Gamma$. We can then construct the following derivation of $\varphi$ from $\Gamma$.
\[
    \begin{array}{rll}
        1.&\alpha\limp\varphi&(\alpha\limp\varphi\in\Gamma)\\
        2.&\beta\limp\varphi&(\beta\limp\varphi\in\Gamma)\\
        3.&(\alpha\limp\varphi)\limp((\beta\limp\varphi)\limp((\alpha\lor\beta)\limp\varphi))&(\hbox{Axiom 8})\\
        4.&(\beta\limp\varphi)\limp((\alpha\lor\beta)\limp\varphi)&(\hbox{MP on (1) \& (3)})\\
        5.&(\alpha\lor\beta)\limp\varphi&(\hbox{MP on (2) \& (4)})\\
        6.&\alpha\lor\beta&(\alpha\lor\beta\in\Gamma)\\
        7.&\varphi&(\hbox{MP on (5) \& (6)})\\
    \end{array}
    \]
    Thus, $\Gamma\vdashmb\varphi$, which is a contradiction. Hence, either $v(\alpha)=1$ or $v(\beta)=1$.
   
   Conversely, suppose either $v(\alpha)=1$ or $v(\beta)=1$, i.e.,   either $\alpha\in\Gamma$ or $\beta\in\Gamma$.

   Suppose $\alpha\in\Gamma$. We then have the following derivation of $\alpha\lor\beta$ from $\Gamma$.
   \[
   \begin{array}{rll}
        1.&\alpha\limp(\alpha\lor\beta)&(\hbox{Axiom 6})\\
        2.&\alpha&(\alpha\in\Gamma)\\
        3.&\alpha\lor\beta&(\hbox{MP on (1) \& (2)})
   \end{array}
   \]
   Thus, $\Gamma\vdashmb\alpha\lor\beta$. Now, as $\Gamma$ is closed, $\alpha\lor\beta\in\Gamma$, which implies that $v(\alpha\lor\beta)=1$.
   
  In case $\beta\in\Gamma$, it can be proved using similar arguments, with Axiom 6 replaced by Axiom 7 in the derivation above, that $\Gamma\vdashmb\alpha\lor\beta$, and hence,  $v(\alpha\lor\beta)=1$.
  
    \item Suppose $v(\alpha\land\beta)=1$. Then, by definition of $v$, $\alpha\land\beta\in\Gamma$. We can then construct the following derivation of $\alpha$ from $\Gamma$.
    \[
    \begin{array}{rll}
        1.&(\alpha\land\beta)\limp\alpha&(\hbox{Axiom 4})\\
        2.&\alpha\land\beta&(\alpha\land\beta\in\Gamma)\\
        3.&\alpha&(\hbox{MP on (1) \& (2))}
    \end{array}
    \]
    Thus, $\Gamma\vdashmb\alpha$. Then, as $\Gamma$ is closed, $\alpha\in\Gamma$. So, $v(\alpha)=1$. It can be proved by similar arguments, with Axiom 4 replaced by Axiom 5 in the above derivation, that $\Gamma\vdashmb\beta$ and hence, $v(\beta)=1$.

    Conversely, suppose $v(\alpha)=1=v(\beta)$. Then, $\alpha,\beta\in \Gamma$. The following derivation of $\alpha\land\beta$ can then be constructed from $\Gamma$.
    \[
    \begin{array}{rll}
         1.&\alpha\limp(\beta\limp(\alpha\land\beta))&(\hbox{Axiom 3})\\
         2.&\alpha&(\alpha\in\Gamma)\\
         3.&\beta\limp(\alpha\land\beta)&(\hbox{MP on (1) \& (2)})\\
         4.&\beta&(\beta\in\Gamma)\\
         4.&\alpha\land\beta&(\hbox{MP on (3) \& (4)})
    \end{array}
    \]
    Thus, $\Gamma\vdashmb\alpha\land\beta$, which implies that $\alpha\land\beta\in \Gamma$ since $\Gamma$ is closed. Hence, $v(\alpha\land\beta)=1$.
   
    \item Suppose $v(\alpha\limp\beta)=1$. Then, $\alpha\limp\beta\in \Gamma$ by definition of $v$. Suppose further that $v(\alpha)\neq0$, i.e., $v(\alpha)=1$. Then, $\alpha\in\Gamma$. So, we have the following derivation of $\beta$ from $\Gamma$.
    \[
    \begin{array}{rll}
         1.&\alpha&(\alpha\in\Gamma)\\
         2.&\alpha\limp\beta&(\alpha\limp\beta\in\Gamma)\\
         3.&\beta&(\hbox{MP on (1) \& (2)})
    \end{array}
    \]
    Thus, $\Gamma\vdashmb\beta$, and hence, $\beta\in\Gamma$ as $\Gamma$ is closed. This, in turn, implies that $v(\beta)=1$. Hence, either $v(\alpha)=0$ or $v(\beta)=1$.

    Conversely, suppose $v(\alpha)=0$ or $v(\beta)=1$.

    \textsc{Case 1:} $v(\alpha)=0$.

    Then, $\alpha\notin\Gamma$. Since, $\Gamma$ is maximal relative to $\varphi$ in $\mbst$, this implies that $\Gamma\cup\{\alpha\}\vdashmb\varphi$. So, by the Deduction theorem (Theorem \ref{thm:DedThm}), $\Gamma\vdashmb\alpha\limp\varphi$. Since $\Gamma$ is closed, $\alpha\limp\varphi\in\Gamma$. Now, if possible, suppose $\alpha\limp\beta\notin\Gamma$. Then, again by the same reasoning as above, $(\alpha\limp\beta)\limp\varphi\in\Gamma$. We can then have the following derivation of $\varphi$ from $\Gamma$.
    \[
    \begin{array}{rll}
         1.&\alpha\limp\varphi&(\alpha\limp\varphi\in\Gamma)\\
         2.&(\alpha\limp\beta)\limp\varphi&((\alpha\limp\beta)\limp\varphi\in\Gamma)\\
         3.&(\alpha\limp\varphi)\limp((\alpha\limp\beta)\limp\varphi)\limp((\alpha\lor(\alpha\limp\beta))
         \limp\varphi)&(\hbox{Axiom 8})\\
         4.&((\alpha\limp\beta)\limp\varphi)\limp((\alpha\lor(\alpha\limp\beta))\limp\varphi)&(\hbox{MP on (1) \& (3)})\\
         5.&(\alpha\lor(\alpha\limp\beta))\limp\varphi&(\hbox{MP on (2) \& (4)})\\
         6.&\alpha\lor(\alpha\limp\beta)&(\hbox{Axiom 9})\\
         7.&\varphi&(\hbox{MP on (5) \& (6)})
    \end{array}
    \]
    Thus, $\Gamma\vdashmb\varphi$, which contradicts the assumption that $\Gamma$ is maximal relative to $\varphi$. Hence, $\alpha\limp\beta\in\Gamma$.
    
    \textsc{Case 2:} $v(\beta)=1$
    
    Then, $\beta\in \Gamma$. So, we have the following derivation of $\alpha\limp\beta$ from $\Gamma$.
    \[
    \begin{array}{rll}
         1.&\beta&(\beta\in\Gamma)\\
         2.&\beta\limp(\alpha\limp\beta)&(\hbox{Axiom 1})\\
         3.&\alpha\limp\beta&(\hbox{MP on (1) \& (2)})
    \end{array}
    \]
    Thus, $\Gamma\vdashmb\alpha\limp\beta$, which implies that $\alpha\limp\beta\in\Gamma$ since $\Gamma$ is closed. So, in either case, $\alpha\limp\beta\in\Gamma$, and hence, $v(\alpha\limp\beta)=1$.

    \item Suppose $v(\alpha)=1$. If possible, let $v(\neg\alpha)=1$. So, $\alpha,\neg\alpha\in\Gamma$. Then, we can construct the following derivation of $\varphi$ from $\Gamma$.
    \[
    \begin{array}{rll}
         1.&\alpha&(\alpha\in\Gamma)\\
         2.&\neg\alpha&(\neg\alpha\in\Gamma)\\
         3.&\alpha\limp(\neg\alpha\limp\varphi)&(\hbox{Axiom 10})\\
         4.&\neg\alpha\limp\varphi&(\hbox{MP on (1) \& (3)})\\
         5.&\varphi&(\hbox{MP on (2) \& (4)})
    \end{array}
    \]
    Thus, $\Gamma\vdashmb\varphi$, which contradicts the assumption that $\Gamma$ is maximal relative to $\varphi$. Hence, $v(\neg\alpha)=0$.
    
    \item Suppose $v(\alpha)=0=v(\neg\alpha)$ but $v(\fstar\alpha)=0$. Then, $\alpha,\neg\alpha,\fstar\alpha\notin\Gamma$.
    
    Since $\Gamma$ is maximal relative to $\varphi$, this implies $\Gamma\vdashmb\alpha\limp\varphi,\,\Gamma\vdashmb\neg\alpha\limp\varphi,\,\Gamma\vdashmb\fstar\alpha\limp\varphi$. Then, as $\Gamma$ is closed, $\alpha\limp\varphi,\neg\alpha\limp\varphi,\fstar\alpha\limp\varphi\in\Gamma$. We can then construct the following derivation of $\varphi$ from $\Gamma$.
    \[
    \begin{array}{rll}
         1.&\alpha\limp\varphi&(\alpha\limp\varphi\in\Gamma)\\
         2.&\neg\alpha\limp\varphi&(\neg\alpha\limp\varphi\in\Gamma)\\
         3.&(\alpha \limp\varphi)\limp((\neg\alpha\limp\varphi)\limp((\alpha\lor\neg\alpha)\limp\varphi))&(\hbox{Axiom 8})\\
         4.&(\neg\alpha\limp\varphi)\limp((\alpha\lor\neg\alpha)\limp\varphi)&(\hbox{MP on (1) \& (3)})\\
         5.&(\alpha\lor\neg\alpha)\limp\varphi&(\hbox{MP on (2) \& (4)})\\
         6.&((\alpha\lor\neg\alpha)\limp\varphi)\limp((\fstar\alpha\limp\varphi)\limp((\alpha\lor\neg\alpha\lor\fstar\alpha)\limp\varphi))&(\hbox{Axiom 8})\\
         7.&(\fstar\alpha\limp\varphi)\limp((\alpha\lor\neg\alpha\lor\fstar\alpha)\limp\varphi)&(\hbox{MP on (5) \& (6)})\\
         8.&\fstar\alpha\limp\varphi&(\fstar\alpha\limp\varphi\in\Gamma)\\
         9.&(\alpha\lor\neg\alpha\lor\fstar\alpha)\limp\varphi&(\hbox{MP on (7) \& (8)})\\
         10.&\alpha\lor\neg\alpha\lor\fstar\alpha&(\hbox{Axiom 11})\\
         11.&\varphi&(\hbox{MP on (9) \& (10)})
    \end{array}
    \]
    Thus, $\Gamma\vdashmb\varphi$. This contradicts the assumption that $\Gamma$ is maximal relative to $\varphi$. Hence, if $v(\alpha)=v(\neg\alpha)=0$, then $v(\fstar\alpha)=1$. 
\end{enumerate}
Thus, $v$ is indeed an $\mbst$-valuation.
\end{proof}

\begin{theorem}[Completeness]\label{thm:completeness}
    Suppose $\Gamma\cup\{\varphi\}\subseteq\lang$. Then, $\Gamma\modelsmb\varphi$ implies $\Gamma\vdashmb\varphi$.
\end{theorem}

\begin{proof}
    Suppose $\Gamma\not\vdashmb\varphi$. By Remark \ref{rem:mbst-Tarskian/finitary}, $\mbst$ is Tarskian and finitary. So, by Lemma \ref{lem:Lind-Asser}, there exists $\Delta\supseteq\Gamma$ that is maximal relative to $\varphi$ in $\mbst$. Then, by Lemma \ref{lem:mbst-val}, there is an $\mbst$-valuation $v$ such that $v(\psi)=1$ for all $\psi\in\Delta$, and hence, for all $\psi\in\Gamma$. Now, as $\Delta$ is maximal relative to $\varphi$, $\Delta\not\vdashmb\varphi$. Since $\mbst$ is reflexive, this implies that $\varphi\notin\Delta$. Thus, $v(\varphi)=0$. Hence, $\Gamma\not\modelsmb\varphi$.
\end{proof}

\begin{remark}
    As mentioned in Remark \ref{rem:tt}, the above Completeness theorem and the observations there prove that $\mbst$ satisfies all the conditions of Definition \ref{def:LFU}, and thus, is indeed an LFU.
\end{remark}

\section{Paracomplete probability}
The connection between logic and probability has a long history. More specifically, the links between non-classical thinking and probability have been discussed in \cite{Williams2016}. This article talks about probability theory based on both paraconsistent and paracomplete logics. There are other articles available on paraconsistent probability theories (see \cite{bueno2016paraconsistent} for a discussion on this). As mentioned earlier, we take the paracomplete route in this article. In contrast to similar approaches in \cite{morgan1983probabilistic} and \cite{weatherson2003classical}, where intuitionistic logic has been employed, we use the LFU $\mbst$ as the base logic.

In this section, we attach probabilities to sentences of $\mbst$ much as in Section 3 of \cite{bueno2016paraconsistent}. A discussion on probabilities defined on sets can be found in a later section.

\begin{definition}
    Suppose $\mathcal{S}=\langle\lang,\vdashs\rangle$ is a logic. A \emph{probability function for the logic} $\mathcal{S}$, or an \emph{$\mathcal{S}$-probability function}, is a function $P:\lang\to[0,1]$ such that for any $\varphi,\psi\in\lang$, $P$ satisfies the following conditions.
    \begin{enumerate}[label=(\roman*)]
        \item If $\vdashs\varphi$, then $P(\varphi)=1$. (\emph{Tautologicity})
        \item If $\varphi\vdashs$ (i.e., $\varphi\vdashs\psi$ for all $\psi\in\lang$), then $P(\varphi)=0$. (\emph{Antitautologicity})
        \item If $\psi\vdashs\varphi$, then $P(\psi)\leq P(\varphi)$. (\emph{Comparison})
        \item $P(\varphi\lor\psi)=P(\varphi)+P(\psi)-P(\varphi\land\psi)$. (\emph{Finite additivity})
    \end{enumerate}
\end{definition} 

It is clear that by changing the logic $\mathcal{S}$, we can obtain different probability functions using these same meta-axioms. The case for classical logic is well-known. The intuitionistic case was dealt with in \cite{weatherson2003classical} and the case for the logic $\ci$ was considered in \cite{bueno2016paraconsistent}. In this article, we take the underlying logic to be the LFU $\mbst$. Some consequences of this choice are listed in the following remark.

\begin{remark}\label{rem:ParacompleteProbRules}
Suppose $P$ is an $\mbst$-probability function, where $\mbst=\langle\lang,\vdashmb\rangle$, as before. Then the following can easily be concluded.
\begin{enumerate}[label=(\roman*)]
    \item $P(\alpha\land\neg\alpha)=0$ for any $\alpha\in\lang$.

    To see this, we note that by Axioms 4 and 5 of $\mbst$, $\vdashmb(\alpha\land\neg\alpha)\limp\alpha$ and $\vdashmb(\alpha\land\neg\alpha)\limp\neg\alpha$. Hence, by the Deduction theorem (Theorem \ref{thm:DedThm}), $\alpha\land\neg\alpha\vdash\alpha$ and $\alpha\land\neg\alpha\vdash\neg\alpha$. Now, by Axiom 10 and MP, $\{\alpha,\neg\alpha\}\vdashmb\beta$ for all $\beta\in\lang$. Thus, by transitivity, $\alpha\land\neg\alpha\vdashmb\beta$ for all $\beta\in\lang$. Hence, by antitautologicity, it follows that $P(\alpha\land\neg\alpha)=0$.
    
    \item $P(\alpha\lor\neg\alpha\lor\fstar\alpha)=1$ for any $\alpha\in\lang$.

    This follows from Axiom 11 of $\mbst$ and tautologicity.
    
    \item $P(\alpha\lor\neg\alpha)=P(\alpha)+P(\neg\alpha)$ for any $\alpha\in\lang$. 

    This follows by finite additivity and (i) above.

    \item If $\alpha\vdashmb\beta$ and $\beta\vdashmb\alpha$ for some $\alpha,\beta\in\lang$, i.e., $\alpha,\beta$ are logically equivalent in $\mbst$, then $P(\alpha)=P(\beta)$.

    This follows by comparison. Clearly, this property is not special to $\mbst$-probability functions and can be derived from the definition of a probability function regardless of the logic.
\end{enumerate}
\end{remark}

\begin{lemma}\label{lem:comparison}
    Suppose $P$ is an $\mbst$-probability function, where $\mbst=\langle\lang,\vdashmb\rangle$. Let $\alpha_1,\alpha_2,\beta\in\lang$ such that $P(\alpha_1)=P(\alpha_2)=1$ and $\{\alpha_1,\alpha_2\}\vdashmb\beta$. Then, $P(\beta)=1$.
\end{lemma}

\begin{proof}
    Since $\alpha_1\limp(\alpha_1\lor\alpha_2)$ is an axiom of $\mbst$ (Axiom 6), by the Deduction theorem (Theorem \ref{thm:DedThm}), $\alpha_1\vdashmb\alpha_1\lor\alpha_2$. So, by comparison, $P(\alpha_1)\le P(\alpha_1\lor\alpha_2)$. Now, as $P(\alpha_1)=1$, this implies that $P(\alpha_1\lor\alpha_2)=1$. By finite additivity, $P(\alpha_1\lor\alpha_2)=P(\alpha_1)+P(\alpha_2)-P(\alpha_1\land\alpha_2)$. Then, by substituting in the values of $P(\alpha_1),P(\alpha_2)$, and $P(\alpha_1\lor\alpha_2)$, we get $P(\alpha_1\land\alpha_2)=1$. Now, as $(\alpha_1\land\alpha_2)\limp\alpha_1$ and $(\alpha_1\land\alpha_2)\limp\alpha_2$ are axioms of $\mbst$ (Axioms 4 and 5), by the Deduction theorem again, $\alpha_1\land\alpha_2\vdashmb\alpha_1$ and $\alpha_1\land\alpha_2\vdashmb\alpha_2$. So, since $\{\alpha_1,\alpha_2\}\vdashmb\beta$, by transitivity, $\alpha_1\land\alpha_2\vdashmb\beta$. Hence, by comparison, $P(\alpha_1\land\alpha_2)\le P(\beta)$. So, $P(\beta)=1$.
\end{proof}

\begin{corollary}\label{cor:comparison}
    Suppose $P$ is an $\mbst$-probability function, where $\mbst=\langle\lang,\vdashmb\rangle$. Let $\alpha_1,\ldots$, $\alpha_m,\beta\in\lang$ such that $P(\alpha_1)=\cdots=P(\alpha_m)=1$ and $\{\alpha_1,\ldots,\alpha_m\}\vdashmb\beta$. Then, $P(\beta)=1$.
\end{corollary}

\begin{proof}
    This follows by extending the argument in the above lemma.
\end{proof}

As discussed in \cite{bueno2016paraconsistent}, probabilities are sometimes seen as generalized truth-values. Given a logic $\mathcal{S}=\langle\lang,\vdashs\rangle$ with valuations interpreting elements of $\lang$ in the 2-element set $\{0,1\}$, $\mathcal{S}$-probability functions may be thought of as extending the valuations. The valuations can, in turn, be seen as degenerate $\mathcal{S}$-probability functions. In such a case, the $\mathcal{S}$-probability functions give rise to a \emph{probabilistic semantics} for the logic $\mathcal{S}$, which is different from the standard truth-valued semantics.

\begin{definition}
Suppose $\mathcal{S}=\langle\lang,\vdashs\rangle$ is a logic. 
Then, a \emph{probabilistic semantic relation} for $\mathcal{S}$ $\probsem\,\subseteq\pow(\lang)\times\lang$ is defined as follows. For any $\Gamma\cup\{\varphi\}\subseteq\lang$, $\Gamma\probsem\varphi$, if for every $\mathcal{S}$-probability function $P$, $P(\varphi)=1$ whenever $P(\psi)=1$ for every $\psi\in \Gamma$.
\end{definition}

The rest of this section is focused on establishing that $\mbst$ is complete (sound and complete) with respect to the probabilistic semantics for $\mbst$.

\begin{theorem}
Suppose $\mbst=\langle\lang,\vdashmb\rangle$ and $\probsem\,\subseteq\pow(\lang)\times\lang$ be the probabilistic semantic relation for $\mbst$. Then, $\Gamma\vdashmb\varphi$ iff $\Gamma\probsem\varphi$.
\end{theorem}

\begin{proof}
Suppose $\Gamma\vdashmb\varphi$. Let $P$ be any $\mbst$-probability function such that $P(\psi)=1$ for all $\psi\in\Gamma$.

Now, as $\mbst$ is finitary, there exists a finite $\Gamma_{0}\subseteq\Gamma$ such that $\Gamma_0\vdashmb\varphi$.

If $\Gamma_0=\emptyset$, i.e., $\vdashmb\varphi$, then $P(\varphi)=1$, by tautologicity.

Suppose $\Gamma_0\neq\emptyset$. Let $\Gamma_0=\{\psi_1,\ldots,\psi_m\}$. Since $\Gamma_0\subseteq\Gamma$, $P(\psi_i)=1$ for all $1\leq i\leq m$. Then, by Corollary \ref{cor:comparison}, $P(\varphi)=1$. Thus, $\Gamma\probsem\varphi$.

Conversely, suppose $\Gamma\probsem\varphi$.

Since $\Gamma\probsem\varphi$, for each $\mbst$-probability function $P$, $P(\varphi)=1$ whenever $P(\psi)=1$ for all $\psi\in\Gamma$. So, in particular, this holds for any $\mbst$-probability function $P^\prime$ with $P^\prime(\lang)\subseteq\{0,1\}$. We claim that any such two-valued $\mbst$-probability function $P^\prime$ is an $\mbst$-valuation. This is established as follows. Let $\alpha,\beta\in\lang$.

\begin{enumerate}[label=(\roman*)]
\item $P^\prime(\alpha\lor\beta)=1$ iff $P^\prime(\alpha)=1$ or $P^\prime(\beta)=1$. 

Suppose $P^\prime(\alpha\lor\beta)=1$ but $P^\prime(\alpha),P^\prime(\beta)\neq1$. So, $P^\prime(\alpha)=P^\prime(\beta)=0$. Then, by finite additivity, $P^\prime(\alpha\lor\beta)=P^\prime(\alpha)+P^\prime(\beta)-P^\prime(\alpha\land\beta)$, which implies that $P^\prime(\alpha\land\beta)=-1$. This is not possible. Hence, either $P^\prime(\alpha)=1$ or $P^\prime(\beta)=1$.

Conversely, suppose either $P^\prime(\alpha)=1$ or $P^\prime(\beta)=1$. By similar arguments as in the proof of Lemma \ref{lem:comparison}, $\alpha\vdashmb\alpha\lor\beta$ and $\beta\vdashmb\alpha\lor\beta$. So, by comparison, $P^\prime(\alpha)\leq P^\prime(\alpha\lor\beta)$ and $P^\prime(\beta)\leq P^\prime(\alpha\lor\beta)$. Since either $P^\prime(\alpha)=1$ or $P^\prime(\beta)=1$, $P^\prime(\alpha\lor\beta)=1$.

\item $P^\prime(\alpha\land\beta)=1$ iff $P^\prime(\alpha)=1$ and $P^\prime(\beta)=1$.

Suppose $P^\prime(\alpha\land\beta)=1$. By similar arguments as in the proof of Lemma \ref{lem:comparison}, $\alpha\land\beta\vdashmb\alpha$ and $\alpha\land\beta\vdashmb\beta$. So, by comparison, $P^\prime(\alpha\land\beta)\leq P^\prime(\alpha)$ and $P^\prime(\alpha\land\beta)\leq P^\prime(\beta)$. This implies that $P^\prime(\alpha)=1$ and $P^\prime(\beta)=1$.

Conversely, suppose $P^\prime(\alpha)=P^\prime(\beta)=1$. Then, again as $\alpha\vdashmb\alpha\lor\beta$, by comparison, $P^\prime(\alpha)\leq P^\prime(\alpha\lor\beta)$. So, $P^\prime(\alpha\lor\beta)=1$. By finite additivity, $P^\prime(\alpha\lor\beta)=P^\prime(\alpha)+P
^\prime(\beta)-P^\prime(\alpha\land\beta)$. Now, by substituting the values of $P^\prime(\alpha),P^\prime(\beta),P^\prime(\alpha\lor\beta)$, we get $P^\prime(\alpha\land\beta)=1$.

\item $P^\prime(\alpha\limp\beta)=1$ iff $P^\prime(\alpha)=0$ or $P^\prime(\beta)=1$.

Suppose $P^\prime(\alpha\limp\beta)=1$. By finite additivity, $P^\prime(\alpha\lor(\alpha\limp\beta))=P^\prime(\alpha)+P^\prime(\alpha\limp\beta)-P^\prime(\alpha\land(\alpha\limp\beta))$.

Since $\alpha \lor (\alpha\limp\beta)$ is an axiom of $\mbst$ (Axiom 9), by tautologicity, $P^\prime(\alpha \lor (\alpha\limp\beta))=1$. Thus, $P^\prime(\alpha\land(\alpha\limp\beta))= P^\prime(\alpha)+P^\prime(\alpha\limp\beta)-1=P^\prime(\alpha)$. Now, it is straightforward to see that $\alpha\land(\alpha\limp\beta)\vdashmb\beta$. Hence, by comparison, $P^\prime(\alpha\land(\alpha\limp\beta))\leq P^\prime(\beta)$. Thus, $P^\prime(\alpha)\le P^\prime(\beta)$. So, if $P^\prime(\alpha)\neq0$, i.e., $P^\prime(\alpha)=1$, then $P^\prime(\beta)=1$. Hence, either $P^\prime(\alpha)=0$ or $P^\prime(\beta)=1$.

Conversely, suppose either $P^\prime(\alpha)=0$ or $P^\prime(\beta)=1$.

Suppose $P^\prime(\alpha)=0$. Again as $P^\prime(\alpha\lor(\alpha\limp\beta))=P^\prime(\alpha)+P^\prime(\alpha\limp\beta)-P^\prime(\alpha\land(\alpha\limp\beta))$ and $P^\prime(\alpha\lor(\alpha\limp\beta))=1$, we have $P^\prime(\alpha\limp\beta)=P^\prime(\alpha\land(\alpha\limp\beta))-P^\prime(\alpha)+1$.

Now, since $P^\prime(\alpha)=0$ and $\alpha\land(\alpha\limp\beta)\vdashmb\alpha$, by comparison, $P^\prime(\alpha\land(\alpha\limp\beta))=0$. Thus, $P^\prime(\alpha\limp\beta)=1$.
 
On the other hand, suppose $P^\prime(\beta)=1$. Now, as $\beta\limp(\alpha\limp\beta)$ is an axiom of $\mbst$ (Axiom 1), by the Deduction theorem (Theorem \ref{thm:DedThm}), $\beta\vdashmb\alpha\limp\beta$. Hence, by comparison, $P^\prime(\alpha\limp\beta)=1$.

\item If $P^\prime(\alpha)= 1$, then $P^\prime(\neg\alpha)=0$.

Suppose $P^\prime(\alpha)=1$ and $P^\prime(\neg\alpha)\neq0$. Then, $P^\prime(\neg\alpha)=1$. Now, by Remark \ref{rem:ParacompleteProbRules} (iii), $P^\prime(\alpha\lor\neg\alpha)=P^\prime(\alpha)+P^\prime(\neg\alpha)=2$. This is not possible. Hence, $P^\prime(\neg\alpha)=0$.

\item If $P^\prime(\alpha)=0$ and $P^\prime(\neg \alpha)=0$, then $P^\prime(\fstar\alpha)=1$.

Suppose $P^\prime(\alpha)=P^\prime(\neg\alpha)=0$ but $P^\prime(\fstar\alpha)\neq1$. Then, $P^\prime(\fstar\alpha)=0$.

Now, by the same arguments as in the proof of Lemma \ref{lem:comparison}, $\fstar\alpha\vdashmb\alpha\lor\neg\alpha\lor\fstar\alpha$, and hence, by comparison, $P^\prime(\alpha\lor\neg\alpha\lor\fstar\alpha)=0$. However, by Remark \ref{rem:ParacompleteProbRules} (ii), $P^\prime(\alpha\lor\neg\alpha\lor\fstar\alpha)=1$. This is a contradiction. Thus, $P^\prime(\fstar\alpha)=1$.
\end{enumerate}

Now, it is straightforward to see that any $\mbst$-valuation is a two-valued $\mbst$-probability function. Thus, the $\mbst$-valuations are the same as the two-valued $\mbst$-probability function. So, from $\Gamma\probsem\varphi$, we can now conclude that $\Gamma\modelsmb\varphi$. Hence, by Theorem \ref{thm:completeness}, $\Gamma\vdashmb\varphi$.
\end{proof}

We next discuss the following \emph{theorem of total paracomplete probability} generalizing the classical theorem of total probability. This also dualizes the theorem of total paraconsistent probability in \cite{bueno2016paraconsistent}.

\begin{theorem}\label{thm:totalparacompleteprob}
Suppose $\mbst=\langle\lang,\vdashmb\rangle$, as before. Let $\alpha,\beta\in\lang$ and $P$ be any $\mbst$-probability function. Then, 
$P(\beta)=P(\beta\land\alpha)+P(\beta\land\neg\alpha)+P(\beta\land\fstar\alpha)-P(\beta\land(\alpha\lor\neg\alpha)\land\fstar\alpha)$.
\end{theorem}

\begin{proof} 
We first claim that $P(\beta)=P(\beta\land(\alpha\lor\neg\alpha\lor\fstar\alpha))$. The following derivation in $\mbst$ shows that $\vdashmb\beta\limp((\alpha\lor\neg\alpha\lor\fstar\alpha)\land\beta)$.
\[
\begin{array}{rll}
1.&\alpha\lor\neg\alpha\lor\fstar\alpha&(\hbox{Axiom 11})\\
2.&(\alpha\lor\neg\alpha\lor\fstar\alpha)\limp (\beta\limp((\alpha\lor\neg\alpha\lor\fstar\alpha)\land\beta))&(\hbox{Axiom 3})\\
3.&\beta\limp((\alpha\lor\neg\alpha\lor\fstar\alpha)\land\beta)&\hbox{(MP on (1) \& (2))}
\end{array}
\]
Hence, by the Deduction theorem (Theorem \ref{thm:DedThm}), $\beta\vdashmb(\alpha\lor\neg\alpha\lor\fstar\alpha)\land\beta$.

Now, $\vdashmb((\alpha\lor\neg\alpha\lor\fstar\alpha)\land\beta)\limp\beta$ (Axiom 5). So, by the Deduction theorem again, $(\alpha\lor\neg\alpha\lor\fstar\alpha)\land\beta\vdashmb\beta$. Hence, by Remark \ref{rem:ParacompleteProbRules} (iv), $P(\beta)=P((\alpha\lor\neg\alpha\lor\fstar\alpha)\land\beta)$.

As noted earlier in Remark \ref{rem:val}, the truth conditions for $\land$ and $\lor$ in $\mbst$ are the same as in CPL$^+$. Hence, $\land$ \& $\lor$ in $\mbst$ obey the same laws as in CPL$^+$. Thus, for any $\varphi,\psi,\theta\in\lang$, and $\mbst$-valuation $v$,
\[
v((\psi\lor\theta)\land\varphi)=v((\varphi\land\psi)\lor(\varphi\land\theta)).
\]
So, by the Completeness theorem (Theorem \ref{thm:completeness}), for any  $\varphi,\psi,\theta\in\lang$,
\[
(\psi\lor\theta)\land\varphi\vdashmb(\varphi\land\psi)\lor(\varphi\land\theta)\quad\hbox{ and }\quad(\varphi\land\psi)\lor(\varphi\land\theta)\vdashmb(\psi\lor\theta)\land\varphi.
\]
Hence, for by Remark \ref{rem:ParacompleteProbRules} (iv), for any $\mbst$-probability function $P$, 
\begin{equation}\label{eqn:1}
   P((\psi\lor\theta)\land\varphi)=P((\varphi\land\psi)\lor(\varphi\land\theta)). 
\end{equation}
By similar arguments as above, we have, for any $\varphi,\psi\in\lang$, and $\mbst$-probability function $P$,
  \begin{equation}\label{eqn:2}
      P(\varphi\land\psi)=P(\psi\land\varphi),\hbox{ and}
  \end{equation}
  \begin{equation}\label{eqn:3}
      P((\varphi\land\psi)\land(\varphi\land\theta))=P(\varphi\land\psi\land\theta).
  \end{equation}
Thus, $P(\beta)=P((\alpha\lor\neg\alpha\lor\fstar\alpha)\land\beta)$
\[
\begin{array}{lll}
     =&P((\beta\land(\alpha\lor\neg\alpha))\lor(\beta\land\fstar\alpha))&\hbox{(By Equation (\ref{eqn:1}))} \\
     =&P(\beta\land(\alpha\lor\neg\alpha))+P(\beta\land\fstar\alpha)-P((\beta\land(\alpha\lor\neg\alpha))\land(\beta\land\fstar\alpha))&\hbox{(By finite additivity)}\\
     =&P(\beta\land(\alpha\lor\neg\alpha))+P(\beta\land\fstar\alpha)-P(\beta\land(\alpha\lor\neg\alpha)\land\fstar\alpha)&\hbox{(By Equation (\ref{eqn:3}))}\\
\end{array}
\]
Now, $P(\beta\land(\alpha\lor\neg\alpha))$
\[
\begin{array}{lll}
     =&P((\alpha\lor\neg\alpha)\land\beta)&\hbox{(By Equation (\ref{eqn:2}))}\\
     =&P((\beta\land\alpha)\lor(\beta\land\neg\alpha))&\hbox{(By Equation (\ref{eqn:1}))}\\
     =&P(\beta\land\alpha)+P(\beta\land\neg\alpha)-P((\beta\land\alpha)\land(\beta\land\neg\alpha))&\hbox{(By finite additivity)}\\
     =&P(\beta\land\alpha)+P(\beta\land\neg\alpha)-P(\beta\land\alpha\land\neg\alpha)&\hbox{(By Equation (\ref{eqn:3}))}
\end{array}
\]
Since $\vdashmb(\beta\land(\alpha\land\neg\alpha))\limp(\alpha\land\neg\alpha)$ (Axiom 5), by the Deduction theorem (Theorem \ref{thm:DedThm}), we have $\beta\land(\alpha\land\neg\alpha)\vdashmb\alpha\land\neg\alpha$. So, by comparison, $P(\beta\land\alpha\land\neg\alpha)\le P(\alpha\land\neg\alpha)$. By Remark \ref{rem:ParacompleteProbRules} (i), $P(\alpha\land\neg\alpha)=0$. Thus, $P(\beta\land\alpha\land\neg\alpha)=0$. So, $P(\beta\land(\alpha\lor\neg\alpha))=P(\beta\land\alpha)+P(\beta\land\neg\alpha)$. Hence, 
\[
P(\beta)=P(\beta\land\alpha)+P(\beta\land\neg\alpha)+P(\beta\land\fstar\alpha)-P(\beta\land(\alpha\lor\neg\alpha)\land\fstar\alpha).
\]
The above equation can also be rearranged into following using the equations (\ref{eqn:1}) and (\ref{eqn:2}) in the suitable order the required number of times.
\[
P(\beta)=P(\beta\land\alpha)+P(\beta\land\neg\alpha)+P(\beta\land\fstar\alpha)-P((\beta\land\alpha\land\fstar\alpha)\lor(\beta\land\neg\alpha\fstar\alpha)).
\]
\end{proof}
\begin{remark}
The above theorem goes to show how the missing evidence case plays a crucial role in probabilities based on $\mbst$. In the case of classical probability, the law of excluded middle holds, i.e., there is no missing evidence case. This becomes a special case of the above when there is no missing evidence for an $\alpha$. In such a situation, $P(\fstar\alpha)=0$. Then, as $\beta\land\fstar\alpha\vdashmb\fstar\alpha$ and $\beta\land(\alpha\lor\neg\alpha)\land\fstar\alpha\vdashmb\fstar\alpha$, by comparison, $P(\beta\land\fstar\alpha)\le P(\fstar\alpha)=P(\beta\land(\alpha\lor\neg\alpha)\land\fstar\alpha)=0$. Hence, the equation of total probability in the above theorem reduces to the standard equation of total probability, i.e., $P(\beta)=P(\beta\land\alpha)+P(\beta\land\neg\alpha)$.
\end{remark}

We can now use the above theorem to obtain a paracomplete version of Bayes' rule which allows the updation of probabilities in light of new information. In the current paracomplete case, this updation can take place in the presence of undeterminedness or conclusive evidence backing all pieces of information.

As usual, the conditional probability of a statement $\alpha$, given statement $\beta$, is defined as follows.
\[
P(\alpha\mid\beta)=\dfrac{P(\alpha\land\beta)}{P(\beta)},\quad\hbox{provided }P(\beta)\neq0.
\]
Then, the classical Bayes' rule is stated as follows.
\[
P(\alpha\mid\beta)=\dfrac{P(\beta\mid\alpha)\cdot P(\alpha)}{P(\beta)},\quad\hbox{provided }P(\beta)\neq0.
\]
$P(\alpha\mid\beta)$ here denotes the posterior probability of $\alpha$ -- the probability of $\alpha$ after $\beta$ is observed, while $P(\alpha)$ denotes the probability of $\alpha$ prior to this. Then, the following theorem follows from the above theorem of total paracomplete probability, Theorem \ref{thm:totalparacompleteprob}. This may be called the \emph{paracomplete Bayes' conditionalization rule}.

\begin{theorem}\label{thm:Bayes}
Suppose $\mbst=\langle\lang,\vdashmb\rangle$, as before. Let $\alpha,\beta\in\lang$ and $P$ be any $\mbst$-probability function. If $P(\alpha),P(\neg\alpha),P((\alpha\lor\neg\alpha)\land\fstar\alpha)$ and $P(\beta)$ are all non-zero, then the following equality holds.
    \[
    P(\alpha\mid\beta)=\dfrac{P(\beta\mid\alpha)\cdot P(\alpha)}{P(\beta\mid\alpha)\cdot P(\alpha)+P(\beta\mid\neg\alpha)\cdot P(\neg\alpha)+P(\beta\mid\fstar\alpha)\cdot P(\fstar\alpha)-K}\,,
    \]
    where $K=P(\beta\mid((\alpha\lor\neg\alpha)\land\fstar\alpha))\cdot P((\alpha\lor\neg\alpha)\land\fstar\alpha)$.
\end{theorem}

\begin{remark}
    The paracomplete Bayes' rule is dual to the one obtained in \cite{bueno2016paraconsistent} for the paraconsistent case. However, the equation obtained here is not as compact as in there. This is, however, expected as undeterminedness leads to more possibilities.
\end{remark}

\begin{example}\label{exm:paracompleteprob}
    A \emph{low-sensitivity diagnostic test} is one that is not very effective at detecting a disease. Such tests often produce weak or inconclusive evidence, leading to indeterminate outcomes rather than clearly positive or negative results. Because of this, even after testing, diagnosis and management may remain uncertain. This situation is distinct from false positives or false negatives.

Suppose a population is being surveyed for a disease using a low-sensitivity test such that, if it yields a positive result, then chances that a  person truly has the disease is 36\%, and the chances are 90\% for not having the disease if the result is negative. Let the following be the abbreviations for the events.

\begin{itemize}
    \item $D$: the event that a person truly has the disease;
    \item $A$: the event that the test result is positive.
\end{itemize}
Then $\neg A$ denotes the event that the test result is negative. The assumptions above can now be translated as follows.
\begin{itemize}
    \item$P(D|A)=0.36$; and
    \item $P(\neg D|\neg A)=0.9$.
\end{itemize}
Suppose the test outcomes in the population are as follows.
\begin{itemize}
    \item $P(A)=0.4$; and 
    \item $P(\neg A)=0.55$.
\end{itemize}
Moreover, suppose 7\% of the test results are so weak (either positive or negative) that they are treated as negligible evidence. Let the overall prevalence of the disease in the population be 23\% and the probability that a person has the disease when the test result is negligible is 1.9\%. Thus, we have 
\begin{itemize}
    \item $P(D)=0.23$;
    \item $P((A\lor \neg A)\land \fstar A))=0.07$; and
    \item $P(D\land (A\lor\neg A)\land\fstar A)=0.019$.
\end{itemize}
Now, by finite additivity,
\[
1=P(A\lor\neg A\lor\fstar A)=P(A)+P(\neg A)+P(\fstar A)-P((A\lor\neg A)\land\fstar A)=0.4+0.55+P(\fstar A)-0.07.
\]
So, $P(\fstar A)=0.12$. Next, since the event $D$ is classical as a person either has the disease or not, 
\begin{itemize}
    \item $P(D\mid\neg A)=1-P(\neg D\mid \neg A)=1-0.9=0.1$; and 
    \item $P(\neg D\mid A)=1-P(D\mid A)=1-0.36=0.64$.
\end{itemize}
Then, by the theorem of total paracomplete probability (Theorem \ref{thm:totalparacompleteprob}), 
\[
\begin{array}{rcl}
     P(D)&=&P(D\land A)+P(D\land\neg A)+P(D\land\fstar A)-P(D\land(A\lor\neg A)\land\fstar A)\\
     \hbox{i.e., } P(D)&=&P(D\mid A)\cdot P(A)+P(D\mid\neg A)\cdot P(\neg A)+P(D\land\fstar A)-P(D\land(A\lor\neg A)\land\fstar A)\\
     \hbox{i.e., }0.23&=&0.36\cdot0.4+0.1\cdot0.55+P(D\land\fstar A)-0.019.
\end{array}
\]
So, $P(D\land\fstar A)=0.05$. Moreover, by the paracomplete Bayes' rule (Theorem \ref{thm:Bayes}), we have 
\[
P(A\mid D)=\dfrac{P(D\mid A)\cdot P(A)}{P(D)}=\dfrac{0.36\cdot0.4}{0.23}\approx  63\%.
\]
If instead of calculating $P(A\mid D)$ using the undetermined component, we use one of the following cases, then the values of $P(A\mid D)$ obtained are shown below. We know that $P(A)$ and $P(\neg A)$ add up to only 0.95 leaving a gap of 5\%.
\begin{enumerate}
\item  If we increase $P(A)$ by $5\%$ (increasing the weight of false positives), then $P(A)=0.45$ and $P(\neg A)=0.55$. By classical Bayes' rule, we obtain
\[
P(A\mid D)=\dfrac{P(D\mid A)\cdot P(A)}{P(D\mid A)\cdot P(A)+P(D\mid\neg A)\cdot P(\neg A)}=\dfrac{0.36\cdot0.45}{0.36\cdot0.45+0.1\cdot0.55}\approx75\%.
\]
\item If we increase $P(\neg A)$ by 5\% (increasing the weight of false negatives), then, $P(A)=0.4$ and $P(\neg A)=0.6$. Then, again by classical Bayes' rule, we have
\[
P(A\mid D)=\dfrac{0.36\cdot0.4}{0.36\cdot0.4+0.1\cdot0.6}\approx 71\%.
\]

\item If we raise both $P(A)$ and $P(\neg A)$ by 2.5\% (distributing the 5\% gap equally), then, $P(A)=0.425$ and $P(\neg A)=0.575$. In this case, we get
\[
P(A\mid D)=\dfrac{0.36\cdot0.425}{0.36\cdot0.425+0.1\cdot0.575}\approx73\%.
\]
\end{enumerate}
We observe that the accuracy of a less reliable test is better reflected in the form of a lower accuracy rate when calculated using the paracomplete Bayes's rule with the undetermined part taken into consideration.
\end{example}

\section{Paracomplete Probability Spaces}
In this final section we touch upon the discussion about probability on sets. A detailed history of the different competing approaches to probability theory can be found in \cite{bueno2016paraconsistent}. As mentioned there, the notion of probability on sets was introduced by Kolmogorov in his classic book in German, which has been translated to English as \cite{Kolmogorov1933}. This approach is connected to measure theory and is usually preferred by mathematicians, statisticians and engineers. However, the probability on sentences approach remains favored by logicians and philosophers.

The main goal here is to present a definition of a \emph{paracomplete probability space}. This can be seen as dual to the one in \cite{bueno2016paraconsistent} for the paraconsistent case and shows a mathematical and probabilistic interpretation of undeterminedness.

The classical probability space is defined using a \emph{$\sigma$-algebra} of sets. A \emph{probability measure} is then defined on the elements of the $\sigma$-algebra. The mathematical details are as follows.

Given a set $\Omega\neq\emptyset$, a \emph{$\sigma$-algebra} on $\Omega$ is a collection $\Sigma\subseteq\pow(\Omega)$ such that the following conditions are satisfied.
    \begin{enumerate}[label=(\roman*)]
        \item $\emptyset,\Omega\in\Sigma$.
        \item If $A\in\Sigma$, then $A^c\in\Sigma$ ($A^c$ denotes the complement of the set $A$).
        \item If $A_1,A_2,\ldots\in\Sigma$, then $\displaystyle\bigcup_{i\ge1}A_i\in\Sigma$, i.e., $\Sigma$ is closed under countable union. (Due to condition (ii) above, by De Morgan's law, this implies that $\Sigma$ is also closed under countable intersection.)
    \end{enumerate}
Then, a map $P:\Sigma\to[0,1]$ is called a \emph{probability measure} if it satisfies the following conditions.
\begin{enumerate}[label=(\roman*)]
    \item $P(\emptyset)=0$ and $P(\Omega)=1$.
    \item If $A_1,A_2,\ldots\in\Sigma$ are pairwise disjoint, then $P\left(\displaystyle\bigcup_{i\ge1}A_i\right)=\displaystyle\sum_{i\ge1}P(A_i)$.
\end{enumerate}
The structure $(\Omega,\Sigma,P)$ is then called a \emph{(classical) probability space}. $\Omega$ is referred to as the \emph{sample space} (the collection of all possible outcomes) and the elements of $\Sigma$ are called \emph{events}. The connection between probability on events and probability on sentences in the classical case is then established via algebraic means. The process is discussed briefly in \cite{bueno2016paraconsistent}.

It turns out that for the paracomplete situation, working with a $\sigma$-algebra is not tenable. More precisely, we cannot work with a collection of events that is closed under complements. This is expected as the classical negation, which corresponds to the operation of taking complements of sets, behaves differently from a paracomplete negation. We define below an alternative of a $\sigma$-algebra, which we call a \emph{$\sigma_p$-algebra}. 

\begin{definition}
    Suppose $\Omega$ is a nonempty set. A \emph{$\sigma_p$-algebra} on $\Omega$ is a collection $\Sigma\subseteq\pow(\Omega)$ such that the following conditions are satisfied.
    \begin{enumerate}[label=(\roman*)]
        \item $\emptyset,\Omega\in\Sigma$
        \item If $A_1,A_2\in\Sigma$, then $A_1\cap A_2\in\Sigma$, i.e., $\Sigma$ is closed under finite intersection.
        \item If $A_1,A_2,\ldots\in\Sigma$, then $\displaystyle\bigcup_{i\ge1}A_i\in\Sigma$, i.e., $\Sigma$ is closed under countable union.
        \item $^\odot:\Sigma\to\Sigma$ and $\diam:\Sigma\to\Sigma$ are two operations on $\Sigma$ such that, for each $A\in\Sigma$,
        \begin{enumerate}[label=(\alph*)]
        \item $A^\odot\cap A=\emptyset$, i.e., $A^\odot\subseteq A^c$, and
        \item $\diam A\cap A^c=A^c\setminus A^\odot$.
    \end{enumerate}
    \end{enumerate}
\end{definition}

The following example shows that a $\sigma_p$-algebra need not be a $\sigma$-algebra.

\begin{example}
    Let $\Omega=\{1,2\}$ and $\Sigma=\{\emptyset,\Omega,K\}$, where $
    K=\{1\}$. The maps $^\odot:\Sigma\to\Sigma$ and $\diam:\Sigma\to\Sigma$ are defined as follows. $A^\odot=\emptyset$ and $\diam A=\Omega$ for all $A\in\Sigma$. 

    Clearly, $\Sigma$ is a $\sigma_p$-algebra on $\Omega$. However, $K^c=\{2\}\notin\Sigma$. Hence, $\Sigma$ is not a $\sigma$-algebra on $\Omega$.
\end{example}

We can now define a \emph{paracomplete probability space} as follows. 

\begin{definition}
    A \emph{paracomplete probability space} is a structure $\langle\Omega,\Sigma,\Pi,P_{\mu}\rangle$, where 
    \begin{enumerate}[label=(\roman*)]
        \item $\Omega\neq\emptyset$ is the sample space composed of all possible outcomes; 
        \item $\Sigma\subseteq\pow(\Omega)$ is a $\sigma_p$-algebra of the set of events;
        \item $\Pi\in\Sigma$ is the set of all determined outcomes; and
        \item the map $P_{\mu}: \Sigma\limp [0,1]$ is a probability measure satisfying the following conditions.
        \begin{enumerate}
            \item $P_{\mu}(\Omega)=1$ and $P_{\mu}(\emptyset)=0$, and
            \item if $A_1,A_2,\ldots\in\Sigma$ are pairwise disjoint, then $P_\mu\left(\displaystyle\bigcup_{i\ge1}A_i\right)=\displaystyle\sum_{i\ge1}P_\mu(A_i)$.
        \end{enumerate}
    \end{enumerate}
\end{definition}

\begin{remark}
    In a paracomplete probability space $\langle\Omega,\Sigma,\Pi,P_{\mu}\rangle$, for any $A\in\Sigma$, $\diam A$ is the set of all outcomes in $\Omega$ for which evidence is missing. These typically will include some elements of $A$ and some of $A^c$. The elements in $A\setminus\diam A$ and $A^\odot$ are fully determined.

    A classical probability space $(\Omega,\Sigma, P)$ can now be seen as a special case of a paracomplete probability space $(\Omega,\Sigma,\Pi,P)$, where $\Pi=\Omega$, $\diam A=\emptyset$, and hence, $A^\odot=A^c$, for each $A\in\Sigma$. Thus, in this case, $\Sigma$ is closed under the operation of taking complements, and consequently, a $\sigma$-algebra on $\Omega$.
\end{remark}

The question of equivalence between probability on sentences and probability on sets in the paracomplete case requires further investigation involving random variables, measure theory and other tools. We leave that for the future.

\section{Conclusions and Future Directions}
In this article, we have investigated paracompleteness using a logic of formal underterminedness that we call $\mbst$. This logic is endowed with a paracomplete negation $\neg$, i.e., the law of excluded middle expressed in terms of this negation is no longer a tautology, and an undeterminedness operator $\fstar$. A Hilbert-style presentation and a two-valued semantics for $\mbst$ have been discussed, followed by the proofs of soundness and completeness.

We have next introduced a probability semantics for the LFU $\mbst$ and established the soundness and completeness results with respect to this semantics. Next, a theorem of total paracomplete probability and a paracomplete Bayes' rule have been established. While the classical counterparts of these can be seen as special cases when there is no undeterminedness, these are also dual to the results obtained in the paraconsistent case.

The final section of the article explores a possible definition of a paracomplete probability space, where probability is endowed on sets instead of individual sentences of a logic. The question of whether the probability on sentences and the probability on sets approaches are equivalent in the paracomplete case, however, remains unanswered here. This can be taken up in a future article.

Apart from the above question, one can also try to investigate in similar lines as in here using different LFUs. One could also extend this work to the investigation of paracomplete possibility theory as in the case of the work in the paraconsistent case.

\bibliographystyle{splncs04}
\bibliography{pp}

\end{document}